\documentclass[11pt,reqno,letterpaper]{amsart}

\usepackage{amssymb}
\usepackage{amsmath}
\usepackage{amsthm}
\usepackage{bbm}
\usepackage{mathrsfs}
\usepackage{doi}
\usepackage{graphicx}

\usepackage{bookmark}
\usepackage{hyperref}
\hypersetup{pdfstartview={FitH}}

\numberwithin{equation}{section}

\newtheorem{theorem}{Theorem}
\newtheorem{lemma}[theorem]{Lemma}

\newtheorem{proposition}[theorem]{Proposition}
\newtheorem{conjecture}[theorem]{Conjecture}
\newtheorem{observation}[theorem]{Observation}

\theoremstyle{definition}

\renewcommand{\leq}{\leqslant}
\renewcommand{\geq}{\geqslant}
\renewcommand{\le}{\leqslant}
\renewcommand{\ge}{\geqslant}

\newcommand{\R}{\mathbb{R}}

\newcommand{\Z}{\mathbb{Z}}

\begin{document}

\title[Sharp variational inequalities on finite graphs]{Sharp variational inequalities for the Hardy--Littlewood maximal operator on finite undirected graphs}

\author[C. Gonz\'{a}lez-Riquelme]{Cristian Gonz\'{a}lez-Riquelme}
\address[C.\,G-R.]{Centre de Recerca Matemàtica, Edifici C, Campus Bellaterra, 08193 Bellaterra, Spain}
\email{cgonzalez@crm.cat}

\author[V. Kova\v{c}]{Vjekoslav Kova\v{c}}
\address[V.\,K.]{Department of Mathematics, Faculty of Science, University of Zagreb, Bijeni\v{c}ka cesta 30, 10000 Zagreb, Croatia}
\email{vjekovac@math.hr}

\author[J. Madrid]{Jos\'e Madrid}
\address[J.\,M.]{Department of Mathematics, Virginia Polytechnic Institute and State University,  225 Stanger Street, Blacksburg, VA 24061-1026, USA
}
\email{josemadrid@vt.edu}



\begin{abstract}
We study sharp $p$-variational inequalities for the Hardy--Littlewood maximal operator on complete graphs, answering in the affirmative a question by Feng Liu and Qingying Xue. We also use computational assistance to find sharp constants in $1$-variational inequalities for all connected graphs on at most five vertices and pose a conjecture on the corresponding sharp constants for path graphs. Finally, we construct finite graphs with arbitrarily large $p$-variational constants.
\end{abstract}

\maketitle


\section{Introduction}
The centered Hardy--Littlewood maximal function of $f \in \textup{L}^1_{\textup{loc}}(\mathbb{R}^d)$ is defined by 
\begin{equation*}
    Mf(x) = \sup_{r>0} \frac{1}{\lambda(\textup{B}(x,r))} \int_{\textup{B}(x,r)} |f(y)| \,\textup{d}y,
\end{equation*}
where $\lambda$ denotes the Lebesgue measure in $\mathbb{R}^d$ and $\textup{B}(x,r)$ stands for the Euclidean open ball centered at $x$ with radius $r$.
The {Hardy--Littlewood maximal operator} is one of the fundamental objects in harmonic analysis, and its mapping properties have been studied for more than a century. Beyond the classical inequalities on $\textup{L}^p$-spaces, in the Euclidean setting, a seminal result of Kinnunen \cite{Kinnunen1997} showed that the Hardy--Littlewood maximal operator is bounded on Sobolev spaces $W^{1,p}(\mathbb{R}^d)$ for $1<p\le\infty$, initiating a systematic study at the level of derivatives. Similar endpoint (the case $p=1$) and variation-type questions have since driven a substantial part of the literature, both in the continuous and discrete settings; see \cite{CarneiroSurvey} and the references therein. In particular, Kurka \cite{Kurka2015} proved
\begin{equation}\label{eq:kurka}
\textup{Var} M f \leq C \,\textup{Var} f,
\end{equation}
for any function $f:\R\to\R$ of bounded variation. Adapting Kurka's argument, an analogous result was obtained in the discrete setting by Temur \cite{Temur2017}, who showed that \eqref{eq:kurka} also holds for any function $f:\Z\to\R$ of bounded (discrete) variation. 
While it is a well-known open problem to determine if \eqref{eq:kurka} also holds with the optimal constant $C=1$, similar sharp bounds for the uncentered Hardy--Littlewood maximal operator were obtained by Aldaz and P\'{e}rez L\'{a}zaro \cite{AldazPerez2007} (in the continuous setting) and Bober, Carneiro, Hughes, and Pierce \cite{BoberCarneiroHughesPierce2012} (in the discrete setting).

One can regard $\Z$ as an infinite undirected graph, where only consecutive integers are connected with an edge.
This way, the aforementioned topic naturally leads to the study of the inequality \eqref{eq:kurka} on more general graphs. 
Let $G=(V,E)$ be a locally finite connected undirected graph.
The \emph{ball} around $v$ of radius $r$ is defined as
\[ \textup{B}_G(v,r) := \{ w\in V : d_G(v,w)\leq r \},\quad v\in V, \ r\geq0, \]
where $d_G$ is the \emph{distance function} on $G$, i.e., $d_G(v,w)$ is the number of edges in the shortest path between $v$ and $w$.
The \emph{Hardy--Littlewood} maximal operator is now
\[ (M_G f)(v) := \max_{r\geq 0} \frac{1}{|\textup{B}_G(v,r)|} \sum_{w\in \textup{B}_G(v,r)} |f(w)|,\quad v\in V. \]
The \emph{$p$-variation} is here defined as
\[ \textup{Var}_p f := \Bigl( \sum_{vw\in E} |f(v)-f(w)|^p \Bigr)^{1/p}
= \Bigl( \frac{1}{2} \sum_{\substack{v,w\in V\\ d_G(v,w)=1} } |f(v)-f(w)|^p \Bigr)^{1/p},\quad p\in(0,\infty). \]
Finally, the constant $\mathbf{C}_{G,p}$ is defined as the smallest number from $[0,\infty]$ such that
\begin{equation*}
\textup{Var}_p M_G f \leq \mathbf{C}_{G,p} \textup{Var}_p f
\end{equation*}
holds for every function $f\colon V\to\R$.
In the particular case where $p=1$, we simply write $\textup{Var}$ for $\textup{Var}_1$ and $\mathbf{C}_{G}$ for $\mathbf{C}_{G,1}$.

The problem of determining the exact values of the numbers $\mathbf{C}_{G,p}$ was studied by Liu and Xue \cite{LX21}. Previous bounds for the $\textup{L}^p$-norms of maximal operators on graphs were obtained by Soria and Tradacete \cite{ST16}, who discovered that complete graphs $K_n$ and star graphs $S_n$ (see Figure \ref{fig:fig1}) have distinguished roles in certain extremal properties of maximal operators.
Liu and Xue \cite[Theorem 1.3]{LX21} proved that for $G=K_n$ the following inequalities hold: $1-1/n\leq \mathbf{C}_{K_n,p}<1$ for all $n\geq 3$, $p>0$ and $\mathbf{C}_{K_3,p}={2}/{3}$ for all $p>0$. 
They conjectured \cite[Conjecture 1.1 (i)]{LX21} that for every $n\ge 3$ and every $p\in(0,\infty)$ one has
\[ \mathbf{C}_{K_n,p}\ =\ 1-\frac1n. \]
A substantial amount of progress towards establishing this conjecture was made by the first and third authors in \cite{GRM21}, where they proved that the conjecture holds for all $p>\log_{6}4\approx 0.77$ \cite[Theorem 1]{GRM21}.

In this paper, we give a full answer, proving that the conjecture holds for all $p>0$.
\begin{theorem}\label{thm:complete}
For every integer $n\geq3$, every number $p\in(0,\infty)$, and every real-valued function $f$ defined on the vertices of $K_n$ we have
\begin{equation}\label{eq:thmcomplete}
\textup{Var}_p M_{K_n} f \leq \Bigl(1 - \frac{1}{n}\Bigr) \textup{Var}_p f.
\end{equation}
\end{theorem}

\begin{figure}
\includegraphics[width=0.28\linewidth]{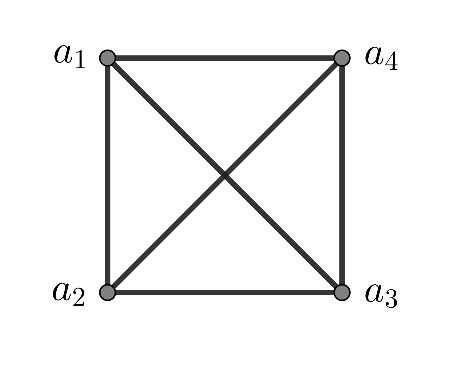}\hspace*{1cm}
\includegraphics[width=0.25\linewidth]{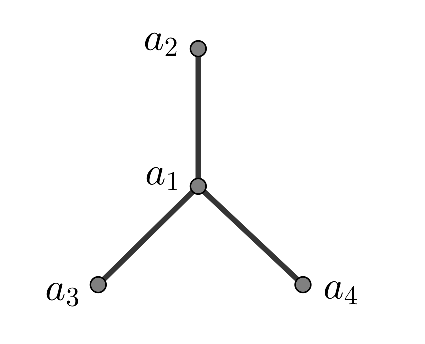}
\caption{Complete graph $K_4$ (left) and star graph $S_4$ (right).}
\label{fig:fig1}
\end{figure}

As noted in \cite{GRM21,LX21}, the constant $1-1/n$ in \eqref{eq:thmcomplete} is sharp, since equality holds for $K_n = \bigl(V,\binom{V}{2}\bigr)$, $V=\{a_1,a_2,\ldots,a_n\}$, $f\colon V\to\R$,
\begin{equation*}
f(a_i) := \begin{cases}
0 & \text{for } 1\leq i\leq n-1,\\
1 & \text{for } i=n,
\end{cases}
\end{equation*}
since then
\[ M_{K_n}f(a_i) = \begin{cases}
1/n & \text{for } 1\leq i\leq n-1,\\
1 & \text{for } i=n.
\end{cases} \]

Liu and Xue also conjectured \cite[Conjecture 1.1 (ii)]{LX21} that the same constant as before is also optimal for the star graph $G=S_n$.
In other words, they asked whether $\mathbf{C}_{S_n,p} = 1 - 1/n$ for $n\geq 3$ and every $p\in(0,\infty)$, after verifying this for $S_3$ and $p\leq 1$.
This was disproved by the first and third authors \cite[Theorem 2 (i)]{GRM21}, by computing
\begin{equation*}
\mathbf{C}_{S_3,p} = \frac{1}{3} \bigl(2^{p/(p-1)}+1\bigr)^{(p-1)/p} \quad\text{for } p>1.
\end{equation*}
Moreover, it has been shown \cite[Theorem 14]{GRM22} that
\[ \mathbf{C}_{S_n,2} = \frac{1}{n} \bigl(n^2-n-1\bigr)^{1/2} \quad\text{for } n\geq 3. \]
What remains is a plausible conjecture that the best constant is still $1-1/n$ for $n\geq 4$ and $0<p\leq1$, whereas \cite[Theorem 2 (ii), (iii)]{GRM21} confirmed this when $1/2\leq p\leq 1$ or $n=4$, and also in a certain larger range of $p\in(0,1)$ depending on $n\geq 5$. However, this line of inquiry is not explored in the present manuscript.  

Let us momentarily turn our attention to the special case $p=1$, hoping to compute sharp $1$-variational constants $\mathbf{C}_G$ of the Hardy--Littlewood maximal operator on some other graphs $G$.
Note that the only connected graphs on $3$ vertices, up to graph isomorphisms, are the two previously mentioned: the complete graph $K_3$ and the star graph $S_3$. Already from \cite{LX21} we know:
\begin{quote}
\emph{All connected graphs $G$ on $3$ vertices have the same constant $\mathbf{C}_G=2/3$.}
\end{quote}

The next simplest candidates to study are undirected graphs on $4$ vertices. There are $6$ such graphs in total. Besides the complete graph $K_4$ and the star graph $S_4$ (from Figure \ref{fig:fig1}), now we have the cycle graph $C_4$ and the path graph $P_4$ (depicted in Figure \ref{fig:fig2}), and then also the paw graph $Y$ and the diamond graph $D$ (illustrated in Figure \ref{fig:fig3}).
Using the basic ideas from linear programming and some computational assistance from the software Mathematica \cite{Mathematica} (see Section \ref{sec:4vertices}), we are able to evaluate the exact constants for each of these graphs, concluding the following.

\begin{theorem}\label{thm:4vertices}
\emph{All connected graphs $G$ on $4$ vertices have the same constant $\mathbf{C}_G=3/4$.}
\end{theorem}

\begin{figure}
\includegraphics[width=0.28\linewidth]{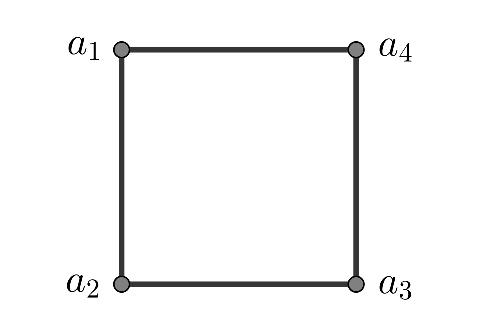}\hspace*{1cm}
\includegraphics[width=0.29\linewidth]{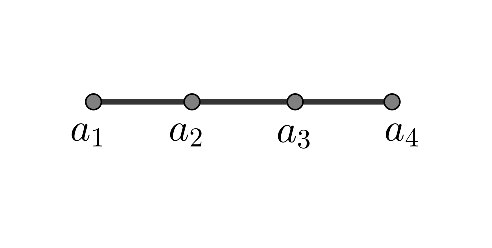}
\caption{Cycle graph $C_4$ (left) and path graph $P_4$ (right).}
\label{fig:fig2}
\end{figure}

\begin{figure}
\includegraphics[width=0.25\linewidth]{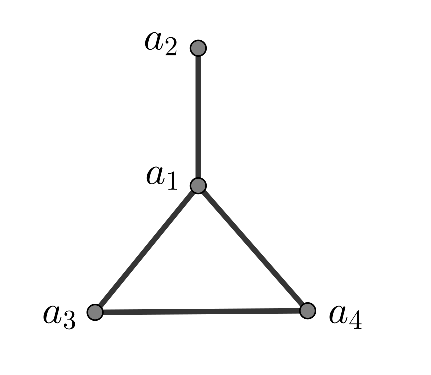}\hspace*{1cm}
\includegraphics[width=0.28\linewidth]{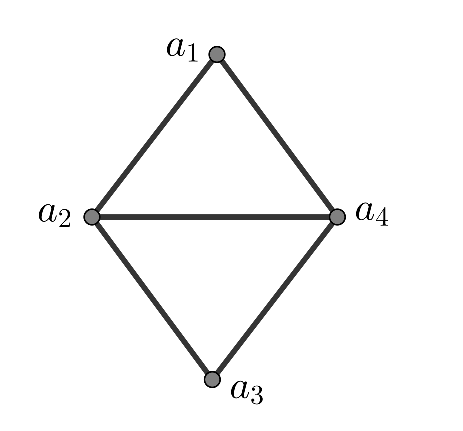}
\caption{Paw graph $Y$ (left) and diamond graph $D$ (right).}
\label{fig:fig3}
\end{figure}

The next challenge is to study the $21$ isomorphism classes of undirected graphs on $5$ vertices. The corresponding optimization problems are solved in Mathematica (see Section \ref{sec:5vertices}) to yield the following.

\begin{observation}\label{obs:5vertices}
\emph{Connected graphs $G$ on $5$ vertices have their sharp constants $\mathbf{C}_G$ in the set $\{4/5, 49/60, 33/40, 5/6, 17/20\}$.}
\end{observation}

A subtle distinction between Theorem \ref{thm:4vertices} and Observation \ref{obs:5vertices} is that the latter is verified using the Mathematica command \verb+Maximize+. While generally reliable, this command does not guarantee finding the global maximum unless the objective function is a polynomial, even though we strongly believe that this is indeed the case here.
We will see that the maximal value $17/20$ is attained for five non-isomorphic graphs, which seems to cause difficulties in a prospective research problem of classifying graphs with a fixed number of vertices that are extremal for $\mathbf{C}_G$.


From the modest amount of data presented above, one might suspect that the $1$-variational constants $\mathbf{C}_G$ never exceed $1$. Our next result shows that, in fact, this is very far from the truth, as the best constants $\mathbf{C}_{G,p}$ can be arbitrarily large, even for every fixed $p\in(0,\infty)$.

\begin{theorem}\label{thm:largeconst}
For every $p,C\in(0,\infty)$ there exists a finite undirected graph $G=(V,E)$ such that $\mathbf{C}_{G,p}>C$.
\end{theorem}

However, the path on $n$ vertices, denoted $P_n$, is a rather special type of graph to which the example from the proof of Theorem \ref{thm:largeconst} (presented in the last section) does not apply. We conjecture the values of their optimal constants $\mathbf{C}_{P_n}$. 

\begin{conjecture}\label{conj:paths}
For every integer $n\geq 3$ the corresponding optimal $1$-variational constant for the path graph $P_n$ equals
\[ \mathbf{C}_{P_n} = 1 - \frac{1}{n}. \]
\end{conjecture}

We have numerical evidence that Conjecture \ref{conj:paths} holds for all $n\leq 10$.
It could be viewed as a graph-theoretic and finitary analogue of the aforementioned open problems on $\Z$ and $\R$, which it implies.
Namely, by embedding an arbitrarily long interval of integers into $P_n$ for sufficiently large $n$, one can easily derive 
\[ \mathbf{C}_{\Z} \leq \limsup_{n\to\infty}\mathbf{C}_{P_n}. \]
In other words, proving $\mathbf{C}_{P_n}\leq 1$ for all $n$ would already suffice to establish the claim $\mathbf{C}_{\Z}\leq 1$ on the integer lattice, and then also on the real line.
Thus, Conjecture \ref{conj:paths}, if true, is expected to be difficult to prove.


\section{Proof of Theorem \ref{thm:complete} on complete graphs}

Note that the cases $p\geq1$ of the theorem were already covered by \cite[Theorem 1 (i), (iii)]{GRM21}.
Here we modify the proof from \cite[\S 2.1.1]{GRM21} so that it applies to all $p\in(0,1)$ and all $n\geq3$.
It is clearly sufficient to consider only nonnegative functions $f$.

For completeness, let us first briefly review the inductive approach from \cite{GRM21}.
The induction basis can be $n=3$, which is simply \cite[Theorem 1.3 (ii)]{LX21}, as has already been commented on.
Take some $n\geq 4$ and suppose that Theorem \ref{thm:complete} holds for the graph $K_{n-1}$.
If the vertices of $G=\bigl(V,\binom{V}{2}\bigr)\cong K_n$ are listed as $V=\{a_1,a_2,\ldots,a_n\}$, then its maximal function is
\[ M_{G} f(a_i) = \max\{f(a_i),m\} \quad\text{for } i=1,2,\ldots,n, \]
where 
\[ m = \frac{1}{n} \sum_{i=1}^{n}f(a_i) \] 
is the average value of $f$.
Without loss of generality, we sort the numbers $f(a_i)$, $1\leq i\leq n$, in increasing order and then take the unique $r\in\{1,\ldots,n\}$ such that
\[ 0 \leq f(a_1) \leq \cdots \leq f(a_{r-1}) < m \leq f(a_r) \leq \cdots \leq f(a_n). \]
The conjectured sharp variational inequality \eqref{eq:thmcomplete} for a given $f\colon V\to[0,\infty)$ now reads
\begin{equation}\label{eq:varcomplete}
\sum_{\substack{i,j\\ i>j\geq r}} \bigl( f(a_i) - f(a_j) \bigr)^p
+ (r-1) \sum_{i=r}^n \bigl( f(a_i) - m \bigr)^p
\leq \Bigl(1-\frac{1}{n}\Bigr)^p \sum_{\substack{i,j\\ i>j}} \bigl( f(a_i) - f(a_j) \bigr)^p.
\end{equation}
Note that the desired inequality \eqref{eq:varcomplete} is actually trivial when $r=1$, as then $f$ needs to be a constant.
Also, its particular case $r=n$ was already shown in \cite[Lemma 10]{GRM21}. Consequently, we can assume $2\leq r\leq n-1$.

Let us apply the induction hypothesis to the graph $G'=\bigl(V',\binom{V'}{2}\bigr)\cong K_{n-1}$ obtained by removing the vertex $a_r$ from $G$ and to the function obtained by restricting $f$ to $V'=V\setminus\{a_r\}$:
\begin{equation}\label{eq:compinducthyp}
\sum_{\substack{i,j\\ i>j\\ i\neq r,\ j\neq r}} \bigl( M_{G'}f|_{V'}(a_i)-M_{G'}f|_{V'}(a_j) \bigr)^p \leq \Bigl(1-\frac{1}{n-1}\Bigr)^p \sum_{\substack{i,j\\ i>j\\ i\neq r,\ j\neq r}} \bigl( f(a_i) - f(a_j) \bigr)^p.
\end{equation}
Observe that the average value of $f|_{V'}$ satisfies 
\[ m' := \frac{1}{n-1} \sum_{\substack{i\\ i\neq r}}f(a_i) \leq m \]
and
\[ M_{G'} f|_{V'}(a_i) = \max\{f(a_i),m'\} \quad\text{for } i\neq r. \]
We claim that 
\begin{equation}\label{eq:comppointwise}
\underbrace{M_{G}f(a_i)-M_{G}f(a_j)}_{\geq0} \leq \underbrace{M_{G'}f|_{V'}(a_i)-M_{G'}f|_{V'}(a_j)}_{\geq0}
\end{equation}
whenever $i>j$, $i\neq r$, $j\neq r$, and verify it case by case.
\begin{itemize}
\item If $i>j>r$, then both sides of \eqref{eq:comppointwise} are equal to $f(a_i)-f(a_j)$.
\item If $r>i>j$, then the left-hand side of \eqref{eq:comppointwise} is zero, while the right-hand side is nonnegative.
\item Finally, if $i>r>j$, then we have
\[ M_{G}f(a_i) = f(a_i) = M_{G'}f|_{V'}(a_i) \]
and
\[ M_{G}f(a_j) = m \geq \max\{f(a_j),m'\} = M_{G'}f|_{V'}(a_j), \]
so \eqref{eq:comppointwise} follows once again.
\end{itemize}
Now, \eqref{eq:compinducthyp}, \eqref{eq:comppointwise}, and $1-1/(n-1) \leq 1-1/n$ together imply
\[ \sum_{\substack{i,j\\ i>j\geq r+1}} \bigl( f(a_i) - f(a_j) \bigr)^p
+ (r-1) \sum_{i=r+1}^n \bigl( f(a_i) - m \bigr)^p
\leq \Bigl(1-\frac{1}{n}\Bigr)^p \sum_{\substack{i,j\\ i>j\\ i\neq r,\ j\neq r}} \bigl( f(a_i) - f(a_j) \bigr)^p. \]
In order to deduce \eqref{eq:varcomplete} from the last estimate, it remains to add the contribution of the vertex $a_r$ after showing
\begin{align*}
& \sum_{i=r+1}^{n} \bigl( f(a_i) - f(a_r) \bigr)^p
+ (r-1) \bigl( f(a_r) - m \bigr)^p \\
& \leq \Bigl(1-\frac{1}{n}\Bigr)^p \Bigl( \sum_{i=r+1}^n \bigl( f(a_i) - f(a_r) \bigr)^p + \sum_{j=1}^{r-1} \bigl( f(a_r) - f(a_j) \bigr)^p \Bigr).
\end{align*}
This last inequality will be a consequence of the following proposition, once the substitutions
\begin{align*}
u & = f(a_r) - m \geq 0, \\
x_i & = f(a_i) - f(a_r) \geq 0, \quad i=r+1,\ldots,n, \\
y_i & = f(a_r) - f(a_i) \geq u, \quad i=1,\ldots,r-1
\end{align*}
are performed.

\begin{proposition}\label{prop:ineq}
For integers $2\leq r< n$, a parameter $p\in(0,1)$, and real numbers
\begin{align*}
x_i\geq 0, & \quad i=r+1,\ldots,n, \\
y_i\geq u\geq 0, & \quad i=1,\ldots,r-1
\end{align*}
satisfying
\[ \sum_{i=r+1}^{n} x_i + nu = \sum_{i=1}^{r-1} y_i, \]
the following inequality holds:
\begin{equation}\label{eq:mainineq}
\sum_{i=r+1}^{n} x_i^p + (r-1)u^p 
\leq \Bigl(1-\frac{1}{n}\Bigr)^p  
\Bigl( \sum_{i=r+1}^{n} x_i^p + \sum_{i=1}^{r-1} y_i^p \Bigr).
\end{equation} 
\end{proposition}

Inequality \eqref{eq:mainineq} was formulated in \cite[Formula (2.17)]{GRM21}, where it was shown only for $p\geq \log_6 4$.
By the previous considerations, and as noted in \cite[Remark 11]{GRM21}, Theorem \ref{thm:complete} follows from the extension of \eqref{eq:mainineq} to all $0<p<1$. 

\begin{proof}[Proof of Proposition \ref{prop:ineq}]
Estimate \eqref{eq:mainineq} can be rewritten as
\begin{equation}\label{eq:mainineqre}
\biggl(1-\Bigl(1-\frac{1}{n}\Bigr)^p\biggr) \sum_{i=r+1}^{n} x_i^p + (r-1)u^p 
\leq \Bigl(1-\frac{1}{n}\Bigr)^p \sum_{i=1}^{r-1} y_i^p.
\end{equation} 
We first apply Karamata's inequality to the concave function $x\mapsto x^p$, just as in \cite{GRM21}, to get
\[ \sum_{i=1}^{r-1} y_i^p \geq \Bigl( \sum_{i=1}^{r-1} y_i - (r-2)u \Bigr)^p + (r-2) u^p
= \Bigl( \sum_{i=r+1}^{n} x_i + (n-r+2)u \Bigr)^p + (r-2) u^p \]
and Jensen's inequality for $x\mapsto x^p$ to obtain
\[ \sum_{i=r+1}^{n} x_i^p 
\leq (n-r)^{1-p} \Bigl( \sum_{i=r+1}^{n} x_i \Bigr)^p. \]
Using these two, estimate \eqref{eq:mainineqre} will follow from
\begin{align*}
\biggl(1-\Bigl(1-\frac{1}{n}\Bigr)^p\biggr) (n-r)^{1-p} \Bigl( \sum_{i=r+1}^{n} x_i \Bigr)^p 
+ \biggl(r-1 - \Bigl(1-\frac{1}{n}\Bigr)^p (r-2)\biggr) u^p & \\
\leq \Bigl(1-\frac{1}{n}\Bigr)^p \Bigl( \sum_{i=r+1}^{n} x_i + (n-r+2)u \Bigr)^p & .
\end{align*}
Dividing by $(1-1/n)^p$ and substituting
\begin{align*}
\alpha = \sum_{i=r+1}^{n} x_i,\quad  & \beta = (n-r+2)u, \\
s = \biggl(\Bigl(\frac{n}{n-1}\Bigr)^p - 1\biggr) (n-r)^{1-p},\quad &
t = \biggl(\Bigl(\frac{n}{n-1}\Bigr)^p (r-1) - r+2\biggr) (n-r+2)^{-p},
\end{align*}
we can rewrite the last inequality as
\[ \alpha^p s + \beta^p t \leq (\alpha+\beta)^p. \]
For its proof it is sufficient to show
\begin{equation}\label{eq:ineqst}
s^{1/(1-p)} + t^{1/(1-p)} \leq 1,
\end{equation} 
because then the proof of the proposition is finalized by H\"{o}lder's inequality for the conjugate exponents $1/p$ and $1/(1-p)$:
\[ \alpha^p s + \beta^p t \leq \bigl( (\alpha^p)^{1/p} + (\beta^p)^{1/p} \bigr)^{p} \bigl( s^{1/(1-p)} + t^{1/(1-p)} \bigr)^{1-p} \stackrel{\eqref{eq:ineqst}}{\leq} (\alpha+\beta)^p. \]

For the aforementioned choices of $s$ and $t$ property \eqref{eq:ineqst} reads
\[ \biggl(\Bigl(\frac{n}{n-1}\Bigr)^p - 1\biggr)^{1/(1-p)} (n-r) 
+ \biggl(\Bigl(\frac{n}{n-1}\Bigr)^p (r-1) - r+2\biggr)^{1/(1-p)} (n-r+2)^{-p/(1-p)} \leq 1 \]
and it will be a consequence of
\begin{equation}\label{eq:ineqst2}
\varphi(x) \leq 1 \quad \text{for } x\in[2,n-1],
\end{equation}
where $n\geq 3$ and $p\in(0,1)$ are fixed, while $\varphi$ is defined as
\begin{align*}
\varphi(x) :=\, & \biggl(\Bigl(\frac{n}{n-1}\Bigr)^p - 1\biggr)^{1/(1-p)} (n-x) \\
& + \biggl(\Bigl(\frac{n}{n-1}\Bigr)^p (x-1) - x+2\biggr)^{1/(1-p)} (n-x+2)^{-p/(1-p)}.
\end{align*}
Observe that
\begin{align*}
\varphi''(x) = & \frac{p}{(1-p)^2} \biggl(\Bigl(\frac{n}{n-1}\Bigr)^p (n+1) - n\biggr)^2 \\
& \times \biggl(\Bigl(\frac{n}{n-1}\Bigr)^p (x-1) - x+2\biggr)^{(2p-1)/(1-p)} (n-x+2)^{(p-2)/(1-p)} >0.
\end{align*}
Consequently, $\varphi$ is convex on the interval $[2,n-1]$, so the verification of \eqref{eq:ineqst2} boils down to
\begin{equation}\label{eq:ineqst3}
\varphi(2) = \frac{(n-2) (n^p-(n-1)^p)^{1/(1-p)} + 1}{(n-1)^{p/(1-p)}} \leq 1
\end{equation}
and
\begin{equation}\label{eq:ineqst4}
\varphi(n-1) = \frac{(n^p-(n-1)^p)^{1/(1-p)} + 3^{-p/(1-p)} ((n-2)n^p-(n-3)(n-1)^p)^{1/(1-p)}}{(n-1)^{p/(1-p)}} \leq 1.
\end{equation}
Inequalities \eqref{eq:ineqst3} and \eqref{eq:ineqst4} are respectively shown in Lemmata \ref{lm:lemma1} and \ref{lm:lemma2} below.
\end{proof}

It remains to establish two auxiliary inequalities.

\begin{lemma}\label{lm:lemma1}
For an integer $n\geq 2$ and a real number $p\in(0,1)$ one has
\[ (n-2) \bigl(n^p-(n-1)^p\bigr)^{1/(1-p)} + 1 \leq (n-1)^{p/(1-p)}. \]
\end{lemma}

\begin{proof}[Proof of Lemma \ref{lm:lemma1}]
Using
\begin{equation}\label{eq:simpleineq}
n^p-(n-1)^p = \int_{n-1}^n \frac{p}{x^{1-p}} \,\textup{d}x \leq \frac{p}{(n-1)^{1-p}}, 
\end{equation}
the inequality will follow from
\[ (n-2) \Bigl(\frac{p}{(n-1)^{1-p}}\Bigr)^{1/(1-p)} + 1 \leq (n-1)^{p/(1-p)}, \]
which simplifies as
\[ p^{1/(1-p)} (n-2) + n-1 \leq (n-1)^{1/(1-p)}. \]
However, this is just
\begin{equation}\label{eq:aux_1}
F_p(n) \geq 0,
\end{equation}
where $F_p\colon[2,\infty)\to\R$ is defined by
\[ F_p(x) = (x-1)^{1/(1-p)} - p^{1/(1-p)} (x-2) - x+1 \]
for a fixed $0<p<1$.
For $x>2$ we have
\[ F_p'(x) = \frac{1}{1-p} (x-1)^{p/(1-p)} - p^{1/(1-p)} - 1
> \frac{p}{1-p} - p^{1/(1-p)}. \]
Positivity of the last expression is equivalent to
\[ -(1-p)\log(1-p) -p\log p > 0, \]
which is easily seen to hold, since its left-hand side is the Shannon entropy of the Bernoulli distribution (i.e., tossing of an unfair coin), which is known to be positive.
Consequently, $F_p$ is strictly increasing on $[2,\infty)$ and \eqref{eq:aux_1} follows from $F_p(2) = 0$.
\end{proof}

\begin{lemma}\label{lm:lemma2}
For an integer $n\geq 2$ and a real number $p\in(0,1)$ one has
\[ \bigl(n^p-(n-1)^p\bigr)^{1/(1-p)} + 3^{-p/(1-p)} \bigl((n-2)n^p-(n-3)(n-1)^p\bigr)^{1/(1-p)} \leq (n-1)^{p/(1-p)}. \]
\end{lemma}

\begin{proof}[Proof of Lemma \ref{lm:lemma2}]
Recalling \eqref{eq:simpleineq}, which also gives
\begin{align*} 
(n-2)n^p-(n-3)(n-1)^p
& = (n-2) \bigl(n^p-(n-1)^p\bigr) + (n-1)^p \\
& \leq (n-2)\frac{p}{(n-1)^{1-p}} + (n-1)^p 
= \frac{p(n-2)+n-1}{(n-1)^{1-p}}, 
\end{align*}
we see that the desired inequality will follow from
\[ \Bigl(\frac{p}{(n-1)^{1-p}}\Bigr)^{1/(1-p)} + 3^{-p/(1-p)} \Bigl(\frac{p(n-2)+n-1}{(n-1)^{1-p}}\Bigr)^{1/(1-p)} \leq (n-1)^{p/(1-p)}, \]
which simplifies as
\[ p^{1/(1-p)} + \Bigl(\frac{pn+n-2p-1}{3^p}\Bigr)^{1/(1-p)} \leq (n-1)^{1/(1-p)}. \]
This can be rewritten as
\begin{equation}\label{eq:aux_2}
G_p(n) \geq 0,
\end{equation}
where $G_p\colon[2,\infty)\to\R$ is defined by
\[ G_p(x) = (x-1)^{1/(1-p)} - \Bigl(\frac{px+x-2p-1}{3^p}\Bigr)^{1/(1-p)} - p^{1/(1-p)} \]
for a fixed $0<p<1$.
When $x>2$, we have
\begin{align*} 
G_p'(x) & = \frac{1}{1-p} (x-1)^{p/(1-p)} - \frac{1}{1-p} \frac{1+p}{3^p} \Bigl(\frac{px+x-2p-1}{3^p}\Bigr)^{p/(1-p)} \\
& \geq \frac{1}{1-p} (x-1)^{p/(1-p)} - \frac{1}{1-p} \frac{1+p}{3^p} \Bigl(\frac{1+p}{3^p}(x-1)\Bigr)^{p/(1-p)} \\
& = \frac{1}{1-p} \biggl(1 - \Bigl(\frac{1+p}{3^p}\Bigr)^{1/(1-p)}\biggr) (x-1)^{p/(1-p)} > 0,
\end{align*}
because
\[ 3^p > e^p > 1+p \]
for $p\in(0,1)$.
Therefore $G_p$ is strictly increasing on $[2,\infty)$ and, in order to verify \eqref{eq:aux_2}, we only still need to check
\[ G_p(2) = 1 - 3^{-p/(1-p)} - p^{1/(1-p)} > 0. \]
This very last inequality immediately follows from
\[ 3^{-p/(1-p)} < e^{-p/(1-p)} = \frac{1}{e^{p/(1-p)}} < \frac{1}{1 + p/(1-p)} = 1-p \]
and
\[ p^{1/(1-p)} < p, \]
so we are done.
\end{proof}

This completes the proof of Proposition \ref{prop:ineq} and thus also that of Theorem \ref{thm:complete}.


\section{Exhaustive study of graphs on \texorpdfstring{$4$}{4} vertices}
\label{sec:4vertices}
We have already commented that $\mathbf{C}_{G}$ is known when $G$ is $K_4$ or $S_4$.
Our goal is to prove Theorem \ref{thm:4vertices} by listing and studying all remaining connected graphs on $4$ vertices (up to graph isomorphisms).

\subsection{Cycle graph \texorpdfstring{$C_4$}{C4}}
If the vertices of $C_4$ are listed as in the left half of Figure \ref{fig:fig2} and we denote $x_i=|f(a_i)|$, then
\[ M_{C_4} f(a_i) = \max\Big\{ x_i, \frac{1}{3}(x_{i-1}+x_i+x_{i+1}), \frac{1}{4}(x_1+x_2+x_3+x_4) \Big\} \]
for $i=1,2,3,4$, where the subtraction/addition of indices is understood modulo $4$, with values in the set of representatives $\{1,2,3,4\}$.
The problem translates to maximizing the objective function
\[ (x_1,x_2,x_3,x_4) \mapsto \sum_{i=1}^{4} |M_{C_4} f(a_i)-M_{C_4} f(a_{i+1})| \]
subject to the constraints
\[ x_1, x_2, x_3, x_4\geq 0, \quad |x_1-x_2| + |x_2-x_3| + |x_3-x_4| + |x_4-x_1| = 1. \]
One could now use the Mathematica command \verb+Maximize+ (and it would confirm the desired result), but we would not be able to rely on that output with absolute certainty, since the objective function is not globally (but only piecewise) a polynomial.

It is thus natural to split the feasible region into subdomains on which both the objective function is linear, and the constraints become linear inequalities/equalities.
By observing the symmetries of $C_4$, we split the problem into the following $9$ linear optimization sub-problems. Each of them is solved autonomously and rigorously using the Mathematica command \verb+LinearOptimization+.

\begin{flushleft}
Maximize $(3 x_1-x_2-x_3-x_4)/2$ subject to $x_1\geq x_2\geq x_3\geq x_4$, $x_1+x_3+x_4\leq 3 x_2$, $2 x_1=1+2 x_4$ returns $2/3$.
\end{flushleft}

\begin{flushleft}
Maximize $(3 x_1-x_2-x_3-x_4)/2$ subject to $x_1\geq x_2\geq x_3\geq x_4$, $x_1+x_3+x_4\geq 3 x_2$, $2 x_1=1+2 x_4$ returns $3/4$.
\end{flushleft}

\begin{flushleft}
Maximize $(3 x_1-x_2-x_3-x_4)/2$ subject to $x_1\geq x_2\geq x_4\geq x_3$, $2 x_1=1+2 x_3$ returns $3/4$.
\end{flushleft}

\begin{flushleft}
Maximize $(3 x_1-x_2-x_3-x_4)/2$ subject to $x_1\geq x_3\geq x_2\geq x_4$, $3 x_3\leq x_1+x_2+x_4$, $2 (x_1-x_2+x_3-x_4)=1$ returns $3/4$.
\end{flushleft}

\begin{flushleft}
Maximize $2 (x_1-x_3)$ subject to $x_1\geq x_3\geq x_2\geq x_4$, $3 x_3\geq x_1+x_2+x_4$, $x_1+x_4\geq 2 x_3$, $2 (x_1-x_2+x_3-x_4)=1$ returns $1/2$.
\end{flushleft}

\begin{flushleft}
Maximize $(5 x_1-7 x_2+5 x_3-3 x_4)/6$ subject to $x_1\geq x_3\geq x_2\geq x_4$, $3 x_3\geq x_1+x_2+x_4$, $x_1+x_4\leq 2 x_3$, $x_1+x_2\leq 2 x_3$, $x_1+x_3+x_4\leq 3 x_2$, $2 (x_1-x_2+x_3-x_4)=1$ returns $1/3$.
\end{flushleft}

\begin{flushleft}
Maximize $2(x_1-x_2+x_3-x_4)/3$ subject to $x_1\geq x_3\geq x_2\geq x_4$, $3 x_3\geq x_1+x_2+x_4$, $x_1+x_4\leq 2 x_3$, $x_1+x_2\leq 2 x_3$, $x_1+x_3+x_4\geq 3 x_2$, $2 (x_1-x_2+x_3-x_4)=1$ returns $1/3$.
\end{flushleft}

\begin{flushleft}
Maximize $(3 x_1-x_2-x_3-x_4)/2$ subject to $x_1\geq x_3\geq x_2\geq x_4$, $3 x_3\geq x_1+x_2+x_4$, $x_1+x_4\leq 2 x_3$, $x_1+x_2\geq 2 x_3$, $x_1+x_3+x_4\leq 3 x_2$, $2 (x_1-x_2+x_3-x_4)=1$ returns $1/2$.
\end{flushleft}

\begin{flushleft}
Maximize $2(2 x_1-x_3-x_4)/3$ subject to $x_1\geq x_3\geq x_2\geq x_4$, $3 x_3\geq x_1+x_2+x_4$, $x_1+x_4\leq 2 x_3$, $x_1+x_2\geq 2 x_3$, $x_1+x_3+x_4\geq 3 x_2$, $2 (x_1-x_2+x_3-x_4)=1$ returns $1/2$.
\end{flushleft}

Therefore
\[ \mathbf{C}_{C_4} = \max \Bigl\{\frac{1}{3},\frac{1}{2},\frac{2}{3},\frac{3}{4}\Bigr\} = \frac{3}{4}. \]
The sharp constant $3/4$ is realized, for instance, when
$x_1=1/2$, $x_2=x_3=x_4=0$.

Note that the splitting into the above linear problems is not minimal: some subproblems (like the first two) can obviously be merged. We were, in fact, manually adding new constraints that split into subproblems whenever Mathematica needed help with simplification of either the objective function or the constraint involving absolute values. Minimality was traded for simplicity. 

\subsection{Path graph \texorpdfstring{$P_4$}{P4}}
If we denote the vertices of $P_4$ as $V=\{a_1,a_2,a_3,a_4\}$, as in the right half of Figure \ref{fig:fig2}, and if we denote $x_i=|f(a_i)|$ for $i=1,2,3,4$, then
{\allowdisplaybreaks\begin{align*}
M_{P_4} f(a_1) & = \max\Big\{ x_1, \frac{1}{2}(x_1+x_2), \frac{1}{3}(x_1+x_2+x_3), \frac{1}{4}(x_1+x_2+x_3+x_4) \Big\}, \\
M_{P_4} f(a_2) & = \max\Big\{ x_2, \frac{1}{3}(x_1+x_2+x_3), \frac{1}{4}(x_1+x_2+x_3+x_4) \Big\}, \\
M_{P_4} f(a_3) & = \max\Big\{ x_3, \frac{1}{3}(x_2+x_3+x_4), \frac{1}{4}(x_1+x_2+x_3+x_4) \Big\}, \\
M_{P_4} f(a_4) & = \max\Big\{ x_4, \frac{1}{2}(x_3+x_4), \frac{1}{3}(x_2+x_3+x_4), \frac{1}{4}(x_1+x_2+x_3+x_4) \Big\}.
\end{align*}}
The objective function is now
\[ (x_1,x_2,x_3,x_4) \mapsto |M_{P_4} f(a_1)-M_{P_4} f(a_2)| + |M_{P_4} f(a_2)-M_{P_4} f(a_3)| + |M_{P_4} f(a_3)-M_{P_4} f(a_4)|, \]
while the constraints are
\[ x_1, x_2, x_3, x_4\geq 0, \quad |x_1-x_2| + |x_2-x_3| + |x_3-x_4| = 1. \]
The same procedure as before splits the problem into $32$ linear optimization subproblems, each of which is easily automatically solved using Mathematica.
In the end, we obtain
$\mathbf{C}_{P_4} = 3/4$.
It is also easy to see that this constant is attained for instance when
$x_1=1$, $x_2=x_3=x_4=0$.

\subsection{Paw graph \texorpdfstring{$Y$}{Y}}
This time, recalling the left half of Figure \ref{fig:fig3},
{\allowdisplaybreaks\begin{align*}
M_{Y} f(a_1) & = \max\Big\{ x_1, \frac{1}{4}(x_1+x_2+x_3+x_4) \Big\}, \\
M_{Y} f(a_2) & = \max\Big\{ x_2, \frac{1}{2}(x_1+x_2), \frac{1}{4}(x_1+x_2+x_3+x_4) \Big\}, \\
M_{Y} f(a_3) & = \max\Big\{ x_3, \frac{1}{3}(x_1+x_3+x_4), \frac{1}{4}(x_1+x_2+x_3+x_4) \Big\}, \\
M_{Y} f(a_4) & = \max\Big\{ x_4, \frac{1}{3}(x_1+x_3+x_4), \frac{1}{4}(x_1+x_2+x_3+x_4) \Big\}.
\end{align*}}
In the same way as before, we maximize $\textup{Var}M_Y f$ with respect to 
\[ x_1, x_2, x_3, x_4\geq 0, \quad |x_1-x_2| + |x_1-x_3| + |x_1-x_4| + |x_3-x_4| = 1. \]
By splitting this task into $29$ simple linear optimization subproblems and solving every one of them, we obtain
$\mathbf{C}_{Y} = 3/4$,
as announced.
One extremizer is
$x_2=1$, $x_1=x_3=x_4=0$.

\subsection{Diamond graph \texorpdfstring{$D$}{D}}
From the right half of Figure \ref{fig:fig3} we can read off
{\allowdisplaybreaks\begin{align*}
M_{D} f(a_1) & = \max\Big\{ x_1, \frac{1}{3}(x_1+x_2+x_4), \frac{1}{4}(x_1+x_2+x_3+x_4) \Big\}, \\
M_{D} f(a_2) & = \max\Big\{ x_2, \frac{1}{4}(x_1+x_2+x_3+x_4) \Big\}, \\
M_{D} f(a_3) & = \max\Big\{ x_3, \frac{1}{3}(x_2+x_3+x_4), \frac{1}{4}(x_1+x_2+x_3+x_4) \Big\}, \\
M_{D} f(a_4) & = \max\Big\{ x_4, \frac{1}{4}(x_1+x_2+x_3+x_4) \Big\}
\end{align*}}
and constraints
\[ x_1, x_2, x_3, x_4\geq 0, \quad |x_1-x_2| + |x_1-x_4| + |x_2-x_3| + |x_2-x_4| + |x_3-x_4| = 1. \]
Following the same procedure one more time, i.e., splitting into $18$ linear optimization subproblems, we get
$\mathbf{C}_{D} = 3/4$
and an extremizer is
$x_2=1/3$, $x_1=x_3=x_4=0$.


\section{Exhaustive study of graphs on \texorpdfstring{$5$}{5} vertices}
\label{sec:5vertices}

\begin{figure}
\includegraphics[width=0.18\linewidth]{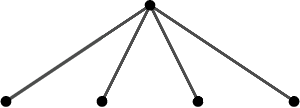}\hspace*{4mm}
\includegraphics[width=0.18\linewidth]{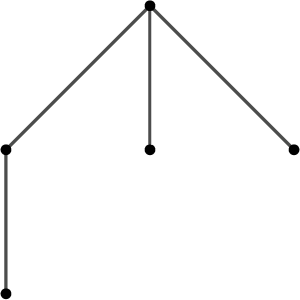}\hspace*{4mm}
\includegraphics[width=0.18\linewidth]{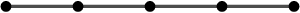}\\[2mm]
\includegraphics[width=0.18\linewidth]{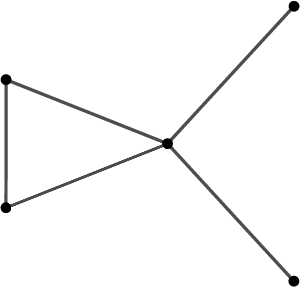}\hspace*{4mm}
\includegraphics[width=0.18\linewidth]{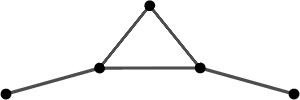}\hspace*{4mm}
\includegraphics[width=0.18\linewidth]{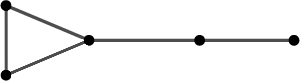}\\[2mm]
\includegraphics[width=0.18\linewidth]{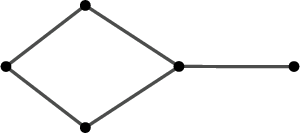}\hspace*{4mm}
\includegraphics[width=0.18\linewidth]{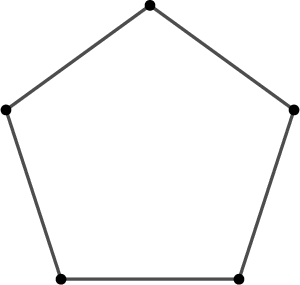}\hspace*{4mm}
\includegraphics[width=0.18\linewidth]{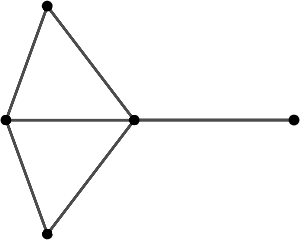}\\[2mm]
\includegraphics[width=0.18\linewidth]{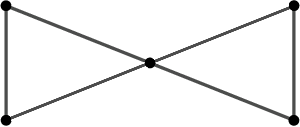}\hspace*{4mm}
\includegraphics[width=0.18\linewidth]{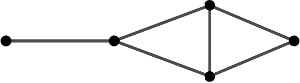}\hspace*{4mm}
\includegraphics[width=0.18\linewidth]{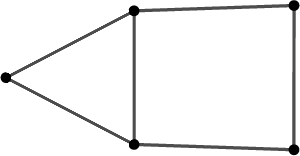}\\[2mm]
\includegraphics[width=0.18\linewidth]{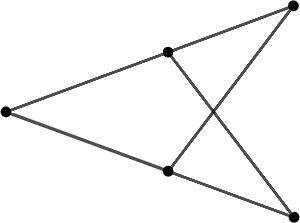}\hspace*{4mm}
\includegraphics[width=0.18\linewidth]{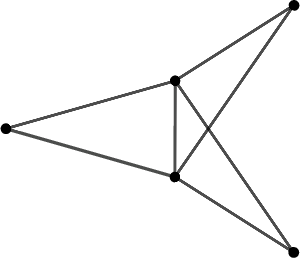}\hspace*{4mm}
\includegraphics[width=0.18\linewidth]{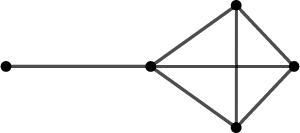}\\[2mm]
\includegraphics[width=0.18\linewidth]{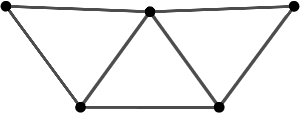}\hspace*{4mm}
\includegraphics[width=0.18\linewidth]{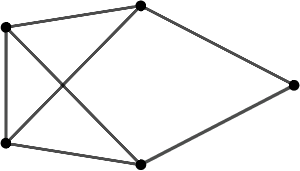}\hspace*{4mm}
\includegraphics[width=0.18\linewidth]{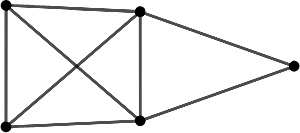}\\[2mm]
\includegraphics[width=0.18\linewidth]{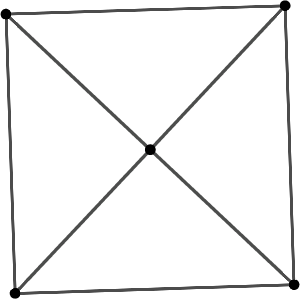}\hspace*{4mm}
\includegraphics[width=0.18\linewidth]{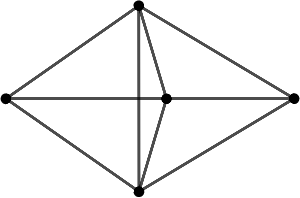}\hspace*{4mm}
\includegraphics[width=0.18\linewidth]{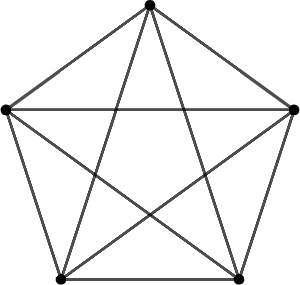}
\caption{Connected graphs on $5$ vertices.}
\label{fig:fig4}
\end{figure}

Our goal is to give strong supportive evidence for Observation \ref{obs:5vertices}.
There are $21$ isomorphism classes of connected graphs on a set of $5$ vertices, depicted in Figure \ref{fig:fig4}. For each of them the Mathematica command \verb+Maximize+ gives a somewhat reliable answer to the underlying constrained optimization problem:

\begin{flushleft}
Maximize $\textup{Var}M_G f$ subject to $f\geq0$, $\textup{Var}f=1$.
\end{flushleft}

When Figure \ref{fig:fig4} is viewed row-wise and from left to right, the optimal constants $\mathbf{C}_G$ of the depicted graphs evaluate respectively as
\[ \frac{4}{5},\frac{4}{5},\frac{4}{5},\frac{4}{5},\frac{17}{20},\frac{4}{5},\frac{17}{20},\frac{4}{5},\frac{33}{40},\frac{4}{5},\frac{17}{20},\frac{33}{40},\frac{17}{20},\frac{4}{5},\frac{4}{5},\frac{33}{40},\frac{17}{20},\frac{49}{60},\frac{5}{6},\frac{33}{40},\frac{4}{5}. \]
We see that the only possible values for $\mathbf{C}_G$ are precisely those claimed before.


\section{Proof of Theorem \ref{thm:largeconst} on large variational constants}

\begin{figure}
\includegraphics[width=0.54\linewidth]{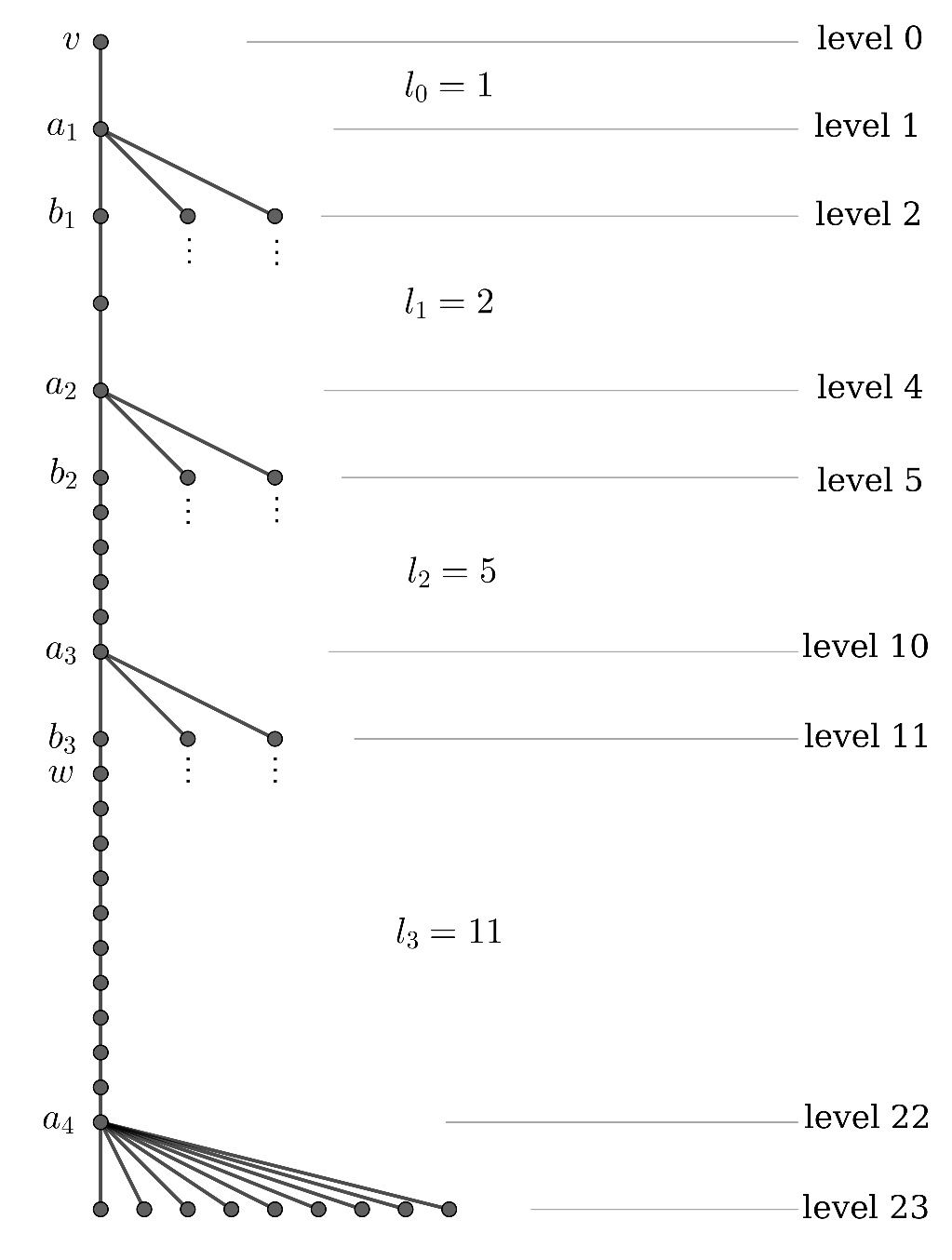}
\caption{Examples of graphs with large $\mathbf{C}_{G,p}$.}
\label{fig:fig5}
\end{figure}

Let $(l_i)_{i=0}^{\infty}$ be the sequence of positive integers defined recursively as
\[ \begin{cases}
& l_0:=1, \\
& l_i := i + \sum_{j=0}^{i-1} l_j \quad\text{for } i\geq1.
\end{cases} \]
Its first several terms are $1,2,5,11,23$, and, in general, by induction we have $l_{n+1}=2l_n+1$, and $l_n=3\cdot2^{n-1}-1$ for all $n\geq 1$. Also fix integers $k\geq2$ and $m>p$.
We will construct $G=(V,E)$ as a rooted tree, with root $v$ and all vertices appearing on levels $0,1,2,3,\ldots,l_{m+1}$. The root vertex $v$ is the only one on level $0$. Precisely on levels labeled $l_i-1$ for $i=1,2,\ldots,m$ every vertex branches into $k$ children, on the level $l_{m+1}-1$ every vertex branches into $k^2$ children, while vertices on all other levels have precisely one offspring each, except for the leaves on the last level.

In other words, the only vertex hanging from the root $v$ is marked as $a_1$, which then branches into $k$ vertices, one of which is denoted by $b_1$. Then we attach to $b_1$ a path of length $l_1=2$ ending with some vertex $a_2$, which then branches into $k$ new vertices, one of which we denote $b_2$. We proceed analogously with every vertex on the same level as $b_1$. Next, to $b_2$ we attach a path of length $l_2=5$ ending in $a_3$, which then branches into $k$ vertices, one of which we call $b_3$, etc. After we finally define the vertices $a_m$ and $b_m$, we attach to $b_m$ a path of length $l_m$ ending in $a_{m+1}$, which branches into $k^2$ leaves of the tree. The vertex immediately under $b_m$ is named $w$. We apply a similar construction to every vertex on the same level as $b_m$. An example of this construction with $k=m=3$ is depicted in Figure \ref{fig:fig5}.

Observe that the ball at $b_m$ of radius $d_G(b_m,v)=d_G(b_m,a_{m+1})$ contains at most $k$ vertices on every tree level and so its cardinality is at most $C_m k$ for some constant $C_m$ depending only on $m$.
On the other hand, the vertex $w$ has more than $k^2$ vertices at distance at most $d(w,v)$; these are the tree leaves.
Thus, taking a function $f\colon V\to\R$ which is $1$ at the root $v$ and $0$ elsewhere, we get
\[ \textup{Var}_p f = 1 \]
and
\[ M_{G}f(b_m) \geq \frac{1}{C_m k},\quad M_{G}f(w) \leq \frac{1}{k^2}. \]
Since there are $k^m$ vertices at the same level as $b_m$, we need to add $k^m$ contributions to the $p$-variation along the edges ``parallel'' to $b_m w$. Consequently,
\[ \textup{Var}_p M_G f \geq \bigl(k^m\bigr)^{1/p} \Bigl(\frac{1}{C_m k}-\frac{1}{k^2}\Bigr), \]
which is, up to a constant, asymptotically equal to $k^{m/p-1}$. It becomes arbitrarily large as $k$ grows to infinity for a fixed $m>p$, making $\mathbf{C}_{G,p}$ as large as needed.
This completes the proof of Theorem \ref{thm:largeconst}.


\section*{Acknowledgments}
C.\,G-R. was supported by the Spanish Ministry of Science, Innovation and Universities, grant PID2023-150984NB-I00. C.\,G-R. is also supported by the Spanish State Research Agency, through the Severo Ochoa and María de Maeztu Program for Centers and Units of Excellence in R\&D (CEX2020-001084-M) and thanks CERCA Programme/Generalitat de Catalunya for institutional support. 
V.\,K. was supported in part by the Croatian Science Foundation under the project HRZZ-IP-2022-10-5116 (FANAP). V.\,K. was also supported in part by the European Union -- NextGenerationEU through the National Recovery and Resilience Plan 2021--2026; institutional grant of University of Zagreb Faculty of Science IK IA 1.1.3. Impact4Math. 
J.\,M. was partially supported by the AMS Stefan Bergman Fellowship and the Simons Foundation Grant $\# 453576$.


\bibliography{SharpVariationGraphs}{}
\bibliographystyle{plainurl}

\end{document}